\documentclass[11pt]{amsart}
\usepackage{latexsym}
\usepackage{amsmath,amsthm,amsfonts,amssymb,graphicx,epsfig,latexsym,color,mathrsfs}
\usepackage{epsf,enumerate}

\newtheorem{thm}{Theorem}[section]
\newtheorem{lemma}[thm]{Lemma}

\newtheorem{prop}[thm]{Proposition}

{}

\theoremstyle{remark}

\newtheorem*{definition*}{Definition}
\newtheorem*{remark*}{Remark}

\def\R{{\mathcal R}}

\def\Z{{\mathbb Z}}

\def\e{{\varepsilon}}

\renewcommand{\phi}{\varphi}
\renewcommand{\bar}{\overline}
\newcommand{\cvn}{\text{cv}_N}
\newcommand{\barcvn}{\bar{\text{cv}}_N}
\newcommand{\FN}{F_N}

\newcommand{\pr}{^{\prime}}

\newcommand{\ssm}{\smallsetminus}

\newcommand{\dom}{\text{dom}}

\begin{document}

\title{Constructing Non-uniquely Ergodic Arational Trees}

\author{Brian Mann and Patrick Reynolds}

\address{\tt Department of Mathematics, University of Utah, 155 S 1400 E, Room 233, Salt Lake City, Utah 84112, USA}
  \email{\tt
  mann@math.utah.edu} 
  \email{\tt
  reynolds@math.utah.edu}

\date{\today}

\begin{abstract}
In this technical note, we adapt an idea of Gabai to construct non-uniquely ergodic, non-geometric, arational trees.
\end{abstract}

\maketitle

\section{Introduction}

Let $F=\FN$ denote the rank-$N$ free group, with $N \geq 4$.  A \emph{factor} is a conjugacy class of non-trivial, proper free factors of $\FN$; a conjugacy class of elements of $\FN$ is \emph{primitive} if any of its representatives generates a representative of a factor.  A \emph{curve} is a one edge $\mathbb{Z}$-splitting of $\FN$ with primitive edge group; every curve is a very small tree.  Two curves $T,T\pr$ are called \emph{disjoint} if there is a two edge simplicial tree $Y$ such that both $T$ and $T\pr$ can be obtained by collapsing the components of an orbit of edges in $Y$ to points; this is the same as having that $l_T + l_{T\pr}$ is a length function for a very small tree, namely $Y$.  

A \emph{measure} on a tree $T$ is a collection of finite Borel measures $\{\mu_I\}$, where $I$ runs over finite arcs of $T$, that is invariant under the $\FN$-action and is compatible with restriction to subintervals; if $T$ has dense orbits, then the set of measures on $T$ is a finite dimensional convex cone \cite{Gui00}.  The measure of a finite subtree $K\subseteq T$ is the sum of measures of the arcs in any partition of $K$ into finitely many arcs.  The \emph{co-volume} of $T$ is the infimum of measures of a finite forests $K$ such that $T \subseteq \cup_{g \in \FN} gK$.  If $H \leq \FN$ is finitely generated and does not fix a point of $T$, then there is a unique minimal $H$ invariant subtree $T_H \subseteq T$; define the co-volume of $H$, denoted $covol(H)$, to be the co-volume of $T_H$, for $T$ fixed.  A tree $T$ is called \emph{arational} if $covol(F\pr)>0$ for every factor $F\pr$; this is the same as having every factor act freely and simplicially on some invariant subtree of $T$, see \cite{R12, BR12}.  An arational tree $T$ is called \emph{non-uniquely ergodic} if there are non-homothetic measures $\mu$ and $\mu\pr$ for $T$; see \cite{Gui00, R10b, CHL07} for more about measures on trees.  Our main result is:

\vspace{.3cm}

\noindent {\bf Theorem:} Let $T, T\pr$ be disjoint curves with neighborhoods $U, U\pr$.  There is a $1$-simplex of non-uniquely ergodic, arational, non-geometric trees with one endpoint in each of $U, U\pr$.  

\vspace{.3cm}

Notions around geometric trees are reviewed in Section \ref{S.Background}.  Examples of non-uniquely ergodic arational trees dual to measured foliations on surfaces are well-known, but all such trees are geometric; R. Martin contructed one example of a non-uniquely ergodic tree that is geometric and of \emph{Levitt type} \cite{Mar97}.  Our procedure gives the first examples of non-uniquely ergodic, arational trees that are non-geometric.  

\subsection{Analogy with Gabai's Construction and Outline of Proof}

Our proof of the Theorem is an adaptation of an idea of Gabai, and there are two main technical steps; before explaining those, we recall the proof of the following:

\begin{thm}\cite[Theorem 9.1]{Gab08}\label{T.Gabai}
 Let $\Sigma$ be a $k$-simplex of disjoint curves in $\mathcal{PML}$, and $U_1,\ldots,U_{k+1}$ be neighborhoods of the extreme points.  There is a $k$-simplex $\Sigma\pr$ of non-uniquely ergodic minimal and filling laminations with extreme points in the $U_j$'s.
\end{thm}

\noindent Theorem \ref{T.Gabai} follows at once by induction after showing:

\begin{prop}\label{P.Gab}\cite{Gab08}
 Let $\alpha_1,\ldots,\alpha_r \in \mathcal{PML}(S_{g,p})$ be a collection of disjoint curves; let $U_j$ be a neighborhood of $\alpha_j$, and let $c$ be a curve.  There are disjoint curves $\alpha_1\pr,\ldots,\alpha_r\pr$ so that $\alpha_j\pr \in U_j$ with neighborhoods $U_j\pr \subseteq U_j$, such that $i(\beta,c)>d_c>0$ for any $\beta \in \cup_j U_j\pr$.  
\end{prop}

Here is how Gabai proves Proposition \ref{P.Gab}.  The surface $S$ is cut into pieces $P_1,\ldots,P_k$ by the $\alpha_j$'s, and glueing back along a fixed $\alpha_j$ gives a surface $\sigma$ that is at least as complex as a $4$-punctured sphere or a punctured torus; $\alpha_j$ is essential and non-peripheral on $\sigma$.  A generic choice $\gamma$ of a curve in $\sigma$ intersects $\alpha_j$ and all arcs of $c$ cutting $\sigma$ (to be more concrete, one can apply a high power of a pseudo-Anosov on $\sigma$ to $\alpha_j$ to get $\gamma$); now apply a high power of a Dehn twist in $\alpha_j$ to $\gamma$ to get a curve $\alpha_j\pr$ in $U_j$, and replace $\alpha_j$ with $\alpha_j\pr$.  To ensure positive intersection of every $\alpha_j\pr$ with $c$, one begins by modifying $\alpha_j$, chosen so that either $c$ meets $\sigma$ in an essential arc or else $c=\alpha_j$; after finding $\alpha_j\pr$, move to a boundary component $\alpha_{j\pr}$ of $\sigma$ and continue.  Since $i(c,\alpha_j\pr)>0$, continuity of $i(\cdot, \cdot)$ ensures that we can find a neighborhoods $U_j\pr \subseteq U_j$ such that $i(c,\beta)>0$ for any $\beta \in \cup_j U_j\pr$.

We proceed essentially in the same way, with our ``curves'' serving as surrogates of curves on a surface.  The analogue, for a tree $T$, of having positive intersection with every simple closed curve on a surface is that every factor acts with positive co-volume in $T$ (this is the same as every factor acting freely and simplicially on its minimal subtree of $T$ \cite{R12}), and our analogue of continuity of $i(\cdot, \cdot)$ is the continuity of the Kapovich-Lustig intersection function combined with continuity of the restriction map from the space of very small $\FN$-trees to the space of $H$-trees for $H$ a finitely generated subgroup of $\FN$.  Since our aim is to construct limiting trees that are both arational and non-geometric, we have two versions of Proposition \ref{P.Gab}; these are the two main technical steps in our argument and appear as Propositions \ref{P.AratLimit} and \ref{P.NonGeomLimit} below.  

This paper is organized as follows:  In Section \ref{S.Background} we give relevant background about trees, currents, and laminations.  In Section \ref{S.Arational} we give our analogue of Proposition \ref{P.Gab} that ensures arational limiting trees, while in Section \ref{S.NonGeometric} we give our analogue of Proposition \ref{P.Gab} that ensures non-geometric limiting trees.  Section \ref{S.Main} contains the proof the main result.  

\vspace{.2cm}

\noindent \emph{Acknowledgements: We wish to thank the participants, Mladen Bestvina in particular, of the Fall 2011 ``working group on $Out(\FN)$'' at the University of Utah for pleasant discussions about Lemma 0.}

\section{Background}\label{S.Background}

Fix a basis $\mathscr{B}$ for $\FN$.  Use $\barcvn$ to denote the set of very small $\FN$-trees; $\cvn$ denotes the subset of free and simplicial very small trees; and $\partial \cvn =\barcvn \ssm \cvn$ \cite{CL95, BF94}.  If $F\pr$ is another finite rank free group, we use $cv(F\pr)$ and $\partial cv(F\pr)$ to denote the corresponding spaces of $F\pr$-trees.  If $T$ is a very small tree, then $l_T$ denotes its length function; spaces of trees get the length functions topology, which coincides with the equivariant Gromov-Hausdorff topology.

\subsection{Currents and Laminations}

Use $\partial \FN$ to denote the boundary of some $T_0 \in \cvn$, and put $\partial^2 \FN := (\partial \FN \times \partial \FN \ssm diag.)/\mathbb{Z}_2$; this can be thought of as the space of unoriented geodesic lines in $T_0$.  The obvious action of $\FN$ on $\partial \FN$ gives an action of $\FN$ on $\partial^2 \FN$.  A \emph{lamination} is a non-empty, invariant, closed subset $L \subseteq \partial^2 \FN$.  A \emph{current} is a non-zero invariant Radon measure $\nu$ on $\partial^2 \FN$; the support of a current $\nu$, denoted $Supp(\nu)$, is a lamination; the set of currents gets the weak-* topology and is denoted by $Curr_N$.  There is a continuous funtion, called \emph{intersection}, $$\langle \cdot, \cdot \rangle:\barcvn \times Curr_N \to \mathbb{R}_{\geq 0}$$ that is homogeneous in the first coordinate and linear in the second coordinate \cite{KL09a}.

If $T \in \partial \cvn$, then either $T$ is not free or else $T$ is not simplicial, hence for any $\epsilon > 0$, there is $g \in \FN$ with $l_T(g) < \epsilon$.  Define $$L(T):= \cap_{\epsilon > 0} \overline{\{(g^{-\infty}, g^{\infty})|l_T(g)<\epsilon\}}$$  The set $L(T)$ is a lamination \cite{CHL08b}.  Kapovich and Lustig gave a complete characterization of when a tree and a current have intersection equal to zero: $\langle T, \nu \rangle =0$ if and only if $Supp(\nu) \subseteq L(T)$ \cite{KL10d}.

\subsection{Geometric Trees}

The topology on $\partial \cvn$ is metrizable, and we fix a compatible metric.  We record a lemma that follows immediately from the definition of the Gromov-Hausdorff topology.

\begin{lemma}\label{L.ContsRestrict}
 For any finitely generated $H\leq \FN$, the function $$\cvn \to cv(H):T \mapsto T_{H}$$ is continuous.
\end{lemma}





Considering the Gromov-Hausdorff topology, one gets a function $\cvn \ni T \mapsto x \in T$ that is ``continuous'' in the following sense: given a finite subset $S \subseteq \FN$ and $\epsilon >0$, there is a $\delta >0$ so that if $T\pr$ is $\delta$-close to $T$, then the partial action of $S$ on the convex hull of $Sx_T$ is $\epsilon$-approximated by the partial action of $S$ on the convex hull of $Sx_{T\pr}$; this point is explained in \cite{Sko90}, in particular, see Skora's discussion of Proposition 5.2.  We call this function a \emph{continuous choice of basepoint} on $\cvn$.

We very quickly recall some notions around geometric trees; see \cite{BF95} for details.  Every $T \in \partial \cvn$ admits \emph{resolutions} as follows: fix $x \in T$ and let $B_n$ be the $n$-ball in the Cayley tree for $\FN$ with respect to $\mathscr{B}$.  The partial isometries induced by elements of $\mathscr{B}$ on the convex hull $K(T,x_T,n)$ of $B_nx$ form a pseudo-group, which can be suspended to get a band complex $Y=Y(T,x,n)$, which is dual to a very small tree $T_n$, and $T_n \to T$ as $n \to \infty$.  The geometric trees $T_n$ come with morphisms $f_n:T_n \to T$, and $T$ is geometric if and only if $f_n$ is an isometry for $n>>0$.  The band complex $Y$ decomposes via Imanishi's theorem into a union of finitely many maximal families of parallel compact leaves, called \emph{families}, and finitely many minimal components, which are glued together along singular leaves; each family $\mathcal C$ has a well-defined \emph{width}, denoted $w(\mathcal C)$, see also \cite{GLP94}.  The family $\mathcal C$ is called a \emph{pseudo-annulus} if every leaf contains an embedded copy of $S^1$; the family $\mathcal C$ is called \emph{non-annular} if it is not a pseudo-annulus.

\subsection{Dehn Twists}

Let $T$ be a curve; for simplicity, assume $T/\FN$ is a circle.  Choosing an edge $e$ in $T$ gives $\FN$ the structure of an HNN-extension $\FN=\langle a_1,\ldots,a_{N-1},t,w'|w'=w^t\rangle$, where $w \in \langle a_1,\ldots,a_{N-1}\rangle$.  The subgroup $V=\langle a_1,\ldots,a_{N-1},w'\rangle$ is the stabilizer of one of the endpoints of $e$, and the subgroup $\langle w \rangle$ is the stabilizer of $e$.  The element $t$ is called the stable letter for this HNN-structure.  Given this choice of $e$ one gets a \emph{Dehn twist automorphism} $\tau$ of $\FN$, defined by $\tau(t)=tw$ and $\tau(a_i)=a_i$; the element $w$ is called the \emph{twistor}.  The class of $\tau$ in $Out(\FN)$ does not depend on the choice of $e$ and also is called a Dehn twist automorphism.  Cohen and Lustig prove the following; see \cite[Theorem 13.2]{CL95}.

\begin{prop}\cite{CL95}\label{P.Twist}
 Let $\tau$ be a Dehn twist with twistor $w$ corresponding to a curve $T$.  If $T\pr \in \partial CV_N$ satisfies $l_{T\pr}(w)>0$, then $\lim_{k \to \pm \infty} T\pr\tau^k=T$.
\end{prop}

If $T\pr$ satisfies the hypotheses of Proposition \ref{P.Twist}, then we simply say that $T\pr$ \emph{intersects} $T$.  We call $\tau$ as above the Dehn twist associated to the curve $T$; in light of Proposition \ref{P.Twist}, the ambiguity of replacing $\tau$ with $\tau^{-1}$ is not important.  Notice that if $T\pr$ is a curve that is disjoint from $T$, then $T\pr\tau=T\pr$.  If $\tau$ is the Dehn twist associated to a curve $T$, then we write $\tau=\tau(T)$; dually, we define  $T_{\tau}$ to be the unique curve satisfying $\tau(T_{\tau})=\tau$.

\section{Forcing Arational Limits}\label{S.Arational}

Here is the first part of our adaptation of Gabai's procedure; we use this result to construct arational trees.  Throughout this section, we blur the distinction between a factor and its representatives, arguing with subgroups and their conjugacy classes as needed.

\begin{prop}\label{P.AratLimit}
Let $F$ be a factor, and let $T$ and $T\pr$ be disjoint curves with neighborhoods $T \in U$, $T\pr \in U'$.  There are disjoint curves $T_1, T_1\pr$ with neighborhoods $T_1 \in U_1 \subseteq U$, $T_1\pr \in U_1\pr \subseteq U\pr$, such that for any $S \in U_1 \cup U_1\pr$, $S_F$ is free and simplicial. 
\end{prop}

\begin{proof}
We will do the proof in the case where both $T$ and $T\pr$ are splittings with one loop-edge, and the common refinement is a graph with one vertex and two loop edges. The other cases are similar, are easier, and are left as exercises to the reader.

Let $V$ and $V\pr$ be the vertex groups of $T$ and $T\pr$, respectively, and let $A$ be the vertex group of the refinement. Let $w$ and $w\pr$ be the generators of the edge groups of $T$ and $T\pr$ respectively, and let $t$ and $t\pr$ be the respective stable letters. Let $A = \langle a_1, \ldots, a_{N-3},w\pr,w^t, w^{\prime t\pr} \rangle $ with $w \in \langle a_1, \ldots, a_{N-3}\rangle $, and where $\{a_1, \ldots, a_{N-3}, t,t\pr, w\pr\}$ is a basis for $F_N$.  Note that $V = \langle A,t\pr\rangle $ and $V\pr = \langle A,t\rangle $. 

It is not the case that $F$ can contain $V$ or $V\pr$; indeed, both $V$ and $V\pr$ strictly contain co-rank 1 factors. 


We need to modify $T,T\pr$ to get new curves in $U,U\pr$, respectively, so that $F$ does not intersect the vertex groups of the new curves.  We accomplish this in several steps, each of which removes certain kinds of intersections; we will appeal to Proposition \ref{P.Twist} to move curves into $U,U\pr$.. 

First suppose that $F$ contains (after choosing a conjugacy representative) $\langle a_1, \ldots, a_{N-3}, t\rangle $. Note that $F$ cannot also contain both $w\pr$ and $t\pr$, or else $F = F_N$. Let $\e$ be whichever of these letters is not contained in $F$. Let $f$ be the automorphism that sends $a_1 \mapsto a_1\e$ and that is the identity on the other basis elements. Replace $T$ with the tree $Tf^{-1}$, and replace $a_1$ by $f(a_1)$; note that $Tf^{-1}$ satisfies the hypotheses of the proposition. Furthermore, $F$ does not contain $\langle a_1, \ldots, a_{N-3}, t\rangle $.

\vspace{.2cm}

\noindent {\bf Step 1:} By Howson's Theorem (see \cite{St83}), there are only finitely many conjugacy classes of intersection of $F$ with $\langle a_1, \ldots, a_{N-3}, t\rangle $. Hence, by applying a sufficiently high power of a fully irreducible automorphism on the factor $\langle a_1, \ldots, a_{N-3}, t\rangle $, say $\phi$, and extending $\phi$ to $F_N$ by sending $w\pr \mapsto w\pr$ and $t\pr \mapsto t\pr$, we can guarantee that in the tree $T\phi^{-1} := T_{1/2}$, conjugates of $F$ intersect the vertex group $\phi (V) = V_{1/2} = \langle \phi(a_1), \ldots, \phi(a_{N-3}), w\pr, t\pr, \phi(w^t)\rangle $ only in elements which contain instances of $w\pr$, $t\pr$, and $\phi(w^t)$ (i.e. no element in the intersection can be contained in any subfactor of $\phi(\langle a_1, \ldots, a_{N-3}, t\rangle )$.). 

The edge group of $T_{1/2}$ is generated by $\phi(w)$, whose reduced form must be a word containing some instances of $t$, and hence it is hyperbolic in $T$. Also, by construction of $\phi$, $T_{1/2}$ remains commonly refined with $T\pr$. The vertex group of the refinement is $\phi(A) = A_{1/2}$.

By Proposition \ref{P.Twist}, we can apply a high power of the Dehn twist $\tau=\tau(T)$ to $T_{1/2}$ to move $T_{1/2}$ into $U$; we use $T_{1/2}$ to denote this new curve as well; note that by applying $\tau$ we could not introduce new intersections of $F$ with $V_{1/2}$ that meet $V$ non-trivially, as $V$ is fixed by $\tau$. 

\vspace{.2cm}

\noindent {\bf Step 2:} Now we perturb the tree $T\pr$. Consider an automorphism $\phi$ of the factor $\langle w\pr, t\pr, t\rangle $ which sends $w\pr \mapsto w\pr \phi(t)$, $t\pr \mapsto t\pr \phi(t)$ and $\phi(t) \mapsto \phi(t)$. Extend $\phi$ to $F_N$. Since the intersection of any conjugate of $F$ with $V_{1/2}$ cannot contain $\phi(t)$, it follows that conjugates of $F$ intersect the $A_{1/2}$ only in elements whose reduced form must contain instances of $\phi(w^t)$ (that is, any elements in the intersection of $F$ with $\langle \phi(a_1), \ldots, \phi(a_{N-3}), \phi(w\pr), \phi(w^t), \phi(w^{\prime t\pr})\rangle $ cannot be contained in $\langle \phi(a_1), \ldots, \phi(a_{N-3}), \phi(w\pr), \phi(w^{\prime t\pr})\rangle $). 

Denote the tree $T\pr\phi^{-1}$ by $T\pr_{1/2}$, so $T_{1/2}\pr$ is disjoint from $T_{1/2}$. Denote by $A_{1/2}\pr:=\phi (A_{1/2})$ the vertex group of the common refinement.

\vspace{.2cm}

\noindent {\bf Step 3:} Now we go back to $T_{1/2}$. By the remark at the end of Step 2, if we apply an automorphism $g$ of $V_{1/2}= \langle A\pr_{1/2},\phi(t\pr)\rangle $ by $t \mapsto t \phi(t\pr)$ and extending to $\FN$, in the resulting tree $T_{1/2}g^{-1} = T_1$, no conjugate of $F$ non-trivially intersects the vertex group $V_1=g(V_{1/2})$.

Again, using Proposition \ref{P.Twist} with $\tau=\tau(T_{1/2})$, we can $T_1$ from the previous paragraph into $U$; call the new curve $T_1$ as well. By construction, $F$ does not meet the vertex group of $T_1$ non-trivially.

Repeat the same process for $T\pr_{1/2}$ to obtain a tree $T\pr_1$ in $U'$, with $T_1\pr$ disjoint from $T_1$ and such that $F$ does not non-trivially intersect the vertex group of $T_1\pr$.

To finish we apply Lemma \ref{L.ContsRestrict} along with the fact that $cv(F)$ is open in $\overline{cv}(F)$ to find neighborhoods $U_1,U_1\pr$ with $T_1 \in U_1 \subseteq U$ and $T_1\pr \in U_1\pr \subseteq U'$, as desired.




\end{proof}

\section{Forcing Non-Geometric Limits}\label{S.NonGeometric}

In this section we bring the second part of our adaptation of Gabai's procedure; we use the main result of this section to construct non-geometric trees as limits of curves.  The reader is assumed to be familiar with Rips theory \cite{BF95}.  First, we record the following:

\begin{lemma}\label{L.Stable}
 Suppose that $Y(T,x_T,n)$ contains a non-annular family of width $w$.  For any $\epsilon>0$, there is $\delta>0$ such that if $T\pr$ is $\delta$-close to $T$, then $Y(T\pr,x_{T\pr},n)$ contains a non-annular family of width $w-\epsilon$.
\end{lemma}

\begin{proof}
 Let $\mathcal{C}$ be a non-annular family of width $w$. Suppose $\mathcal{C}$ intersects $K(T,x_T,n)$ in the intervals $b_0, b_1, b_2, \ldots, b_k$. Use $\phi_{i,j}$ to denote the partial isometry, which corresponds to an element of $\mathscr{B}^{\pm}$, that maps $b_i$ to $b_j$. 

 As $\mathcal{C}$ is a family, the interior of each $b_i$ is disjoint from the set of extremal points of $\dom(\phi_b)$, for every partial isometry $\phi_b$ corresponding to $b \in \mathscr{B}^{\pm}$; further, the orbit of each point of $b_i$ is finite and is contained in $\mathcal{C}$, and the length of each $b_i$ is $w$.  Note that by definition of the Gromov-Hausdorff topology, for any $\eta>0$, there is $\delta\pr>0$, such that for $T\pr$ $\delta\pr$-close to $T$, there is a $(1+\eta)$ bi-Lipschitz, $B_n$-equivariant map $f=f(T,T\pr,n)$ from an $\eta$-dense subtree of $K(T,x_T,n)$ onto an $\eta$-dense subtree of $K(T\pr,x_{T\pr},n)$; this uses that the $K(\cdot,\cdot,\cdot)$'s are trees.
 
 Now, choose $\eta$ small enough so that $2k\eta(1+\eta)<<\epsilon$, and let $T\pr$ be $\delta\pr$-close to $T$ with $\delta\pr$ as in the previous paragraph.  Use $b_i\pr$ for the $f$-image of $b_i$, $\phi_{i,j}\pr$ for the corresponding partial isometries of $K(T\pr,x_{T\pr},n)$.  Note that $\phi_{i,i+1}\pr$ is defined on a central segment $I_i\pr$ of $b_i\pr$ of length at least $w/(1+\eta)-2\eta(1+\eta)$ and that $\phi_{i,i+1}\pr(I_i)$ overlaps with $b_{i+1}\pr$ in a central segment of width at least $w/(1+\eta)-4\eta(1+\eta)$.  Hence, there is a central segment $J_0 \subseteq b_0\pr$ of length at least $w/(1+\eta)-2k\eta(1+\eta)$ with $\phi_{0,j}\pr(J_0)$ contained in the central segment of length $w/(1+\eta)-2\eta(1+\eta)$ of $b_j\pr$.  Hence, no $\phi_{0,j}\pr$-image of $J_0$ can meet an extremal point of $\dom(\phi_{i,i+1}\pr)$.
 
 Let $\mathcal{C}\pr$ be the union of leaves in $Y(T\pr,x_{T\pr},n)$ containing the points of $J_0$.  Note that by choosing $T\pr$ $\delta\pr$-close to $T$, we have ensured that $\phi_{0,1}\pr$ is the only partial isometry from $\mathscr{B}^{\pm}$ that is defined on $J_0$, since this is true in $Y(T,x_T,n)$; similarly, $(\phi_{k-1,k}\pr)^{-1}$ is the only element of $\mathscr{B}^{\pm}$ that is defined on $\phi_{0,k}\pr(J_0)$.  Hence, $\mathcal{C}\pr$ is contained in a family, which has width at least $w-\epsilon$, as desired.   
\end{proof} 

\begin{prop}\label{P.NonGeomLimit}
Let $T$ and $T\pr$ be disjoint curves with neighborhoods $T \in U$, $T\pr \in U'$, and let $n \in \mathbb{N}$ be given.  There are disjoint curves $T_1, T_1\pr$ with neighborhoods $T_1 \in U_1 \subseteq U$, $T_1\pr \in U_1\pr \subseteq U\pr$, such that for any $S \in U_1 \cup U_1\pr$, $Y(S,x_S,n)$ contains a non-annular family of width bounded away from zero. 
\end{prop}

\begin{proof}
 We begin with an observation: if $A \in \partial \cvn$ and if $y \in A$ is a point that is fixed by a subgroup $H\leq\FN$, then $H$ fixes a point in $A_n$ if and only if $H^g=\langle h_1,\ldots,h_r\rangle$ for $\{h_1,\ldots,h_r\} \subseteq B_n$, where $g \in \FN$ is chosen so that $H^g$ fixes $gy \in K(A,x_a,1)$ ($H^g$ is cyclically reduced with respect to $\mathscr{B}$).  It follows that if $A$ is the curve with vertex group $V=\langle a_1,\ldots,a_{N-2},w,w^t\rangle$, edge stabilizer $\langle w \rangle$, and stable letter $t$, then $A_n$ contains an edge with non-trivial stabilizer if and only if $w \in H$ and $w^t \in H_t$, with $H,H_t$ fixing points in $A$ and having generating sets contained in $B_n$.  
 
 In light of the discussion in the above paragraph, we proceed along the lines of the proof of Proposition \ref{P.AratLimit}.  For simplicity, we assume that $\mathscr{B}=\{a_1,\ldots,a_{N-3},w\pr,t,t\pr\}$ is a basis for $\FN$ so that the vertex groups $V,V\pr$ of $T,T\pr$ are $V=\langle a_1,\ldots,a_{N-3},w\pr,t\pr,w^t\rangle$, $V\pr=\langle a_1,\ldots,a_{N-3},t, w\pr,{w\pr}^{t\pr}\rangle$; any other basis gives a quasi-isometric word metric, and the below argument is more cumbersome with arbitrary $\mathscr{B}$.  
 
 Replace $t$ with $tu$, where $u$ is a long random word in $\langle w\pr, t\pr\rangle$.  Next, as in Step 1 of the proof of Proposition \ref{P.AratLimit}, find an irreducible automorphism $\phi$ of $\langle a_1,\ldots,a_{N-3},t\rangle$ with large dilitation and extend $\phi$ to $\FN$ in the obvious way.  Use $\psi$ to denote the composition of these two automorphisms.  Apply a high power of $\tau(T)$ to $T\psi$ to get $T_{1/2} \in U$.  Note that the stable letter for $T_1$ has very long length with 
respect to $\mathscr{B}$.
 
 Now, repeat the procedure from the above paragraph to $T\pr$ to get $T_{1/2}\pr$ in $U\pr$ with stable letter having very long length with respect to $\mathscr{B}$.  Now repeat both these operations on $T_{1/2}$ and $T_{1/2}\pr$ to get $T_1$ and $T_1\pr$.  
 
 By the discussion in the first paragraph of this proof, there is $w>0$ so that both $Y(T_1,x_{T_1},n)$ and $Y(T_1\pr, x_{T_1\pr},n)$ contain a non-annular family of width at least $w$.  Choose $\epsilon<<w$, and let $\delta>0$ be as given by Lemma \ref{L.Stable}.  The intersections of the $\delta$-balls around $T_1$, $T_1\pr$ with $U$, $U\pr$ give neighborhoods $U_1$, $U_1\pr$ of $T$, $T\pr$ satisfying the conclusions of the statement.  
\end{proof}

\section{The Proof of the Main Result}\label{S.Main}

We will use Lemma \ref{L.Stable} along with the following characterization of non-geometric trees with dense orbits.  Note that arational trees have dense orbits \cite{R12}.

\begin{lemma}\label{L.NonGeom}
 Let $T \in \partial \cvn$ have dense orbits.  The tree $T$ is non-geometric if and only if $Y(T,x_T,n)$ contains a non-annular family for every $n$.
\end{lemma}

\begin{proof}
 By Imanishi's theorem, if $T$ has dense orbits and is geometric, then for $n>>0$, $Y(T,x_T,n)$ is a union of minimal components.  Further, the space $Y(T,x_T,n)$ can contain an annular family only if $T$ contains a non-degenerate arc with non-trivial stabilizer, which is impossible if $T$ has dense orbits; see, for instance, \cite{LL03}.  
\end{proof}

Having $T_n \to T$ not exact can be thought of as $T$ being non-geometric on the scale $B_n$.  Our interest in Lemmas \ref{L.Stable} and \ref{L.NonGeom} can be paraphrased as follows for trees $T$ with dense orbits: if $T$ is non-geometric on scale $B_n$, then for $T\pr$ close enough to $T$, $T\pr$ also is non-geometric on scale $B_n$; and if $T$ is non-geometric on the scale $B_n$ for every $n$, then $T$ is non-geometric.  Now, we are in position to prove our main result.



\begin{thm}
 Let $T,T\pr$ be disjoint curves with neighborhoods $U,U\pr$.  There is a $1$-simplex of non-uniquely ergodic, arational, non-geometric trees with one extreme point in each of $U$, $U\pr$.
\end{thm}

\begin{proof}
 Enumerate all factors of $\FN$ as $F^1, F^2, \ldots, F^k, \ldots$.  Set $T_0=T$, $T_0\pr=T\pr$ and $U_0=U$, $U_0\pr=U\pr$.  We proceed inductively, defining for $k>0$ $T_k$, $T_k\pr$ with neighborhoods $U_k \subseteq U_l$, $U_k\pr \subseteq U_l\pr$ for $l \leq k$, with $\{F^j\}_{j\leq l}$ acting freely and simplicially on any $S \in U_k \cup U_k\pr$ and with $Y(S,x_S,l)$ containg a non-annular family of width greater than $w(k)>0$ for any $S \in U_k \cup U_k\pr$.  Assume that $T_{k-1}$, $T_{k-1}\pr$ are defined.  To define $T_k$, $T_k\pr$, $U_k$, $U_k\pr$, first apply Proposition \ref{P.AratLimit} to $T_{k-1}$, $T_{k-1}\pr$, $U_{k-1}$, $U_{k-1}\pr$, and then apply Proposition \ref{P.NonGeomLimit} to the result.  By shrinking the neighborhoods from the conclusion of Proposition \ref{P.AratLimit} slightly we can assume that they are contained in compact neighborhoods satisfying the same conclusions.
 
 Note that each $U_k$, $U_k\pr$ is contained in a ball by construction, and the radii of these balls must go to zero (for example, by density of non-arational trees in $\partial \cvn$).  Hence, $T_k$ converge as to some $\lambda \in \partial \cvn$, and $T_k\pr$ converge to some $\lambda\pr$.  By construction, any factor of $\FN$ acts freely and simplicially on $\lambda$ and $\lambda\pr$.  Further, by Lemma \ref{L.NonGeom}, $\lambda$ and $\lambda\pr$ are non-geometric.  
 
 Use $w_k$ to denote the edge stabilizer of $T_k$, and let $\eta_k$ be the counting current corresponding to $w_k$, so $\langle T_k, \eta_k\rangle=0$.  Since $T_k\pr$ is disjoint from $T_k$, we have that $\langle T_k\pr, \eta_k\rangle=0$ as well.  Let $\eta$ be a representative of any accumulation point in projective current of the images of $\eta_k$.  By continuity of $\langle \cdot, \cdot \rangle$, we have that $\langle \lambda, \eta \rangle =0=\langle \lambda\pr, \eta \rangle$.  Further, by the Kapovich-Lustig characterization of zero intersection, we have that $\emptyset \neq Supp(\eta) \subseteq L(\lambda) \cap L(\lambda\pr)$.  
 
 By Theorem 4.4 of \cite{BR12}, we have that $L(\lambda)=L(\lambda\pr)$, and by Theorem II of \cite{CHL07} any convex combination $\alpha l_{\lambda} +(1-\alpha)l_{\lambda\pr}$ is the length function of a very small tree $T_{\alpha}$.  On the other hand, from the definition of $L(\cdot)$, we certainly have that $L(\lambda) \subseteq L(T_{\alpha})$; applying Theorem 4.4 of \cite{BR12} again gives that $T_{\alpha}$ is arational with $L(T_{\alpha})=L(\lambda)$.  Hence, the segment $\{T_{\alpha}|\alpha \in [0,1]\}$ satisfies the conclusion.
\end{proof}

\bibliographystyle{amsplain}
\bibliography{indecompREF.bib}

\end{document}